\documentclass[12pt, leqno]{amsart}
\usepackage[OT2,T1]{fontenc} 
\DeclareSymbolFont{cyrletters}{OT2}{wncyr}{m}{n}
\DeclareMathSymbol{\Sha}{\mathalpha}{cyrletters}{"58}

\usepackage{indentfirst}
\usepackage{amstext}
\usepackage{amsopn}
\usepackage{amsfonts}
\usepackage{amsmath}
\usepackage{latexsym}
\usepackage{amscd}
\usepackage{amssymb}
\usepackage{amsmath}
\usepackage[all,cmtip]{xy}
\usepackage{leftidx}
\usepackage{graphicx}
\usepackage{tikz}
\usepackage{ulem}
\usepackage{hyperref}

=\oddsidemargin \font\teneufm=eufm10 \font\seveneufm=eufm7
\font\fiveeufm=eufm5
\newfam\eufmfam
\textfont\eufmfam=\teneufm \scriptfont\eufmfam=\seveneufm
\scriptscriptfont\eufmfam=\fiveeufm

\let\goth\mathfrak

\def\cO{\mathcal O}

\def\gm{\goth m}

\def\gG{\goth G}
\def\gH{\goth H}

\def\1{\mbox{\bf 1}}

\def\Ouf{\mathrm{Of}}

\DeclareMathOperator{\Aut}{Aut}

\DeclareMathOperator{\Out}{Out} 
\DeclareMathOperator{\Isom}{Isom}

\DeclareMathOperator{\SO}{\rm SO}

\newcommand{\incl}[1][r]
{\ar@<-0.2pc>@{^(-}[#1] \ar@<+0.2pc>@{-}[#1]}

\newtheorem{stheorem}{Theorem}[section]

\newtheorem{slemma}[stheorem]{Lemma}

\newtheorem{sexample}[stheorem]{Example}

\theoremstyle{definition}

\numberwithin{equation}{section}

\def\ZZ{\mathbb{Z}}

\def\gE{\mathfrak{E}}
\def\gF{\mathfrak{F}}
\def\gG{\mathfrak{G}}

\def\cO{\mathcal{O}}

\def\2int{\mathop{2\int}\nolimits}

\def\Spec{\mathop{\rm Spec}\nolimits}

\def\Gal{\mathop{\rm Gal}\nolimits}

\def\Pic{\mathop{\rm Pic}\nolimits}

\def\Aut{\text{\rm{Aut}}}
\def\Out{\text{\rm{Out}}}
\def\sm{\smallskip}

\def\Isom{\mathop{\rm Isom}\nolimits}

\def\resp.{\mathop{\rm resp.}\nolimits}

\def\lgr{\longrightarrow}

\font\math=cmmi10
\def\varpi{\hbox{\math\char'44}}

\def\simlgr{\buildrel\sim\over\lgr}

\def\pa{\S\kern.15em }

\def\un{\uppercase\expandafter{\romannumeral 1}}
\def\deux{\uppercase\expandafter{\romannumeral 2}}
\def\trois{\uppercase\expandafter{\romannumeral 3}}
\def\quatre{\uppercase\expandafter{\romannumeral 4}}
\def\cinq{\uppercase\expandafter{\romannumeral 5}}
\def\six{\uppercase\expandafter{\romannumeral 6}}

\def\hfl#1#2#3{\smash{\mathop{\hbox to#3{\rightarrowfill}}\limits
^{\scriptstyle#1}_{\scriptstyle#2}}}
\def\gfl#1#2#3{\smash{\mathop{\hbox to#3{\leftarrowfill}}\limits
^{\scriptstyle#1}_{\scriptstyle#2}}}

\title[Loop group schemes]{Loop group schemes 
and Abhyankar's lemma \\ 
{\quad} \\
{\tiny \it Sch\'emas en groupes de lacets et 
lemme d'Abhyankar}}

\author[P. Gille]{Philippe Gille}\address{UMR 5208
Institut Camille Jordan - Universit\'e Claude Bernard Lyon 1
43 boulevard du 11 novembre 1918
69622 Villeurbanne cedex - France 
}

\email{gille@math.univ-lyon1.fr}

\date{\today}

\begin{document}

 \begin{abstract}  We define the notion of loop reductive group schemes defined over the localization of a regular henselian ring $A$ at a
 strict normal crossing divisor $D$. We provide a criterion for the existence of  parabolic subgroups of a given type.

\medskip

\noindent{\it R\'esum\'e.} On d\'efinit la  notion de
sch\'emas en groupes r\'eductifs de lacets au-dessus du localis\'e d'un anneau hens\'elien  $A$ en un 
diviseur \`a croisements normaux stricts $D$. 
On \'etablit un crit\`ere  pour qu'un tel sch\'ema en groupes
admette  un sous--sch\'ema en groupes paraboliques d'un type donn\'e.

\smallskip



\end{abstract}

\maketitle


\bigskip

\noindent{\bf Version fran\c caise abr\'eg\'ee}

\smallskip

Soit $A$ un anneau local hens\'elien r\'egulier muni d'un syst\`eme de param\`etres $f_1,\dots, f_r$. 
On note $k$ le corps r\'esiduel de $A$ et $D$ le diviseur $D=\mathrm{div}(f_1)+  
\dots +\mathrm{div}(f_r)$, il est \`a croisements normaux stricts. On pose $X=\Spec(A)$ et $U=
X \setminus D=\Spec(A_D)$. La th\'eorie d'Abhyankar d\'ecrit 
les rev\^etements finis \'etales connexes 
de $U=X \setminus D=\Spec(A_D)$ qui 
sont mod\'er\'ement ramifi\'es le long de $D$
 \cite[XIII.2]{SGA1}.
Un tel objet est domin\'e par un rev\^etement galoisien
de la forme 
$$
B_n=B[T_1^{\pm 1},\dots, T_r^{\pm 1}]/ ( T_1^{n} - f_1, \dots, T_r^{n}- f_r)
$$
o\`u $n$ d\'esigne un entier $\geq 1$ premier
\`a la caract\'eristique de $k$ et $B$ est une  $A$--alg\`ebre galoisienne contenant  une racine  primitive $n$--i\`eme de
  l'unit\'e. Le groupe de Galois $\Gal(B_n/A_D)$
est le produit semi-direct $\mu_n(B)^r \rtimes \Gal(B/A)$
o\`u $\mu_n(B)^r$ agit par multiplication sur 
les $T_1, \dots, T_r$.
Si $G$ est un $A$-sch\'ema en groupes localement de pr\'esentation finie,  un $1$-cocycle 
$z: \Gal(B_n/A_D) \to G(B_n)$
est dit de {\it lacets}
 (loop en anglais) s'il est \`a valeurs dans 
$G(B) \subset G(B_n)$. Cette terminologie est 
inspir\'ee par l'analogie avec 
le cas des polyn\^omes de Laurent
\cite[ch. 3]{GP}.

On note $\widehat X$ l'\'eclat\'e de $X=\Spec(A)$ 
en son point ferm\'e, c'est  un sch\'ema r\'egulier
 \cite[\S 8.1,  th. 1.19]{Liu} et  le diviseur exceptionnel $E \subset \widehat X$ est
 un diviseur de  Cartier  isomorphe \`a $\mathbb{P}^{r-1}_k$.
On note alors   $R=\mathcal{O}_{\widehat X,\eta}$ l'anneau local au point g\'en\'erique  $\eta$ de $E$.  
Cet anneau  $R$ est  de valuation discr\`ete, de
corps des fractions  $K$, et de corps r\'esiduel  
$F=k(E)=k(t_1, \dots, t_{r-1})$ o\`u $t_i$ d\'esigne l'image
de $\frac{f_i}{f_r} \in R$ par l'application de sp\'ecialisation $R \to F$.
On note alors  $v: K^\times \to \ZZ$
la valuation discr\`ete associ\'ee \`a $R$ et $K_v$
le compl\'et\'e de $K$. Le r\'esultat principal de cette note est le suivant.

\smallskip

\noindent{\bf Th\'eor\`eme} (extrait du  th. \ref{thm_fixed}).
{\it On suppose que $G$ agit  sur un $A$--sch\'ema propre et lisse $Z$. 
Soit   $\phi$ un $1$--cocycle de lacets pour $G$.
On note ${_\phi \! Z}/U$ le tordu par $\phi$ de $Z\times_X U$
  Alors les assertions suivantes sont \'equivalentes:

\sm

(a)  $({_\phi \!Z)}(U) \not \not = \emptyset$;

\sm 

(b) $({_\phi \!Z})(K_v) \not = \emptyset$.
}

\medskip

C'est assez proche d'un r\'esultat sur les polyn\^omes de Laurent \cite[\S , thm. 7.1]{GP}. L'application principale concerne le cas d'un sch\'ema en groupes r\'eductifs
de lacets. 
Par d\'efinition, un $U$-sch\'ema en groupes r\'eductifs 
$G$ est {\it de lacets} si il est isomorphe \`a un tordu
de sa forme d\'eploy\'ee $G_0$ par un  $1$--cocycle de lacets \`a valeurs dans le sch\'ema en groupes
des automorphismes $\Aut(G_0)$.  
On applique alors le r\'esultat ci-dessus 
 \`a des $A$--sch\'emas de  sous-groupes paraboliques
de $G_0$ d'un type donn\'e (th. \ref{thm_parabolic}). On en d\'eduit par exemple que  si $G$ est un 
$U$-sch\'ema en groupes r\'eductifs 
 {\it de lacets}, alors 
 $G$ admet 
un  $U$--sch\'ema en groupes de Borel si et seulement si le $K_v$--sch\'ema en groupes $G_{K_v}$ 
est quasi-d\'eploy\'e.
Plus g\'en\'eralement l'isotropie de $G$ est  contr\^ol\'ee par l'indice de Tits de $G_{K_v}$.

\section{Introduction}\label{section_intro}
In the reference \cite{GP}, we investigated a theory of loop reductive group 
schemes over the ring of Laurent polynomials $k[t_1^{\pm 1}, \dots ,
t_n^{\pm 1}]$. Using Bruhat-Tits' theory, this permitted to relate
what the 
study of those group schemes to that of reductive algebraic groups
over the field of iterated Laurent series $k((t_1)) \dots ((t_n))$.
The main issue of this note is to start a similar approach
for reductive group schemes defined over the  localization  $A_D$ of 
a regular henselian ring $A$ at a  strict normal crossing divisor $D$
and to relate with  algebraic groups defined over a natural field
associated to $A$ and $D$, namely the completion $K_v$ of
the fraction field $K$ with respect to the valuation arising
from the blow-up of $\Spec(A)$ at its maximal ideal.
The example which connects the two viewpoints is 
$k[[t_1,\dots, t_n]][\frac{1}{t_1}, \dots \frac{1}{t_n}]$
where 
$K_v \cong k\bigl( \frac{t_1}{t_n} , \dots, \frac{t_1}{t_{n-1}})((t_n))$.

After defining the notion of loop reductive group schemes in this setting,
we show that for this class of group schemes, the existence
of parabolic subgroups over the localization $A_D$ is controlled by the
 parabolic subgroups over $K_v$ (Theorem \ref{thm_parabolic}).

\medskip

\noindent{\bf Acknowledgements.}  We thank R. Parimala for
sharing her insight about the presented results.
Finally we thank the referee for a careful reading of the manuscript. 
 
\bigskip  

\section{Tame fundamental group}

\subsection{Abhyankar's lemma} Let $X=\Spec(A)$ be a regular local scheme (not assumed henselian at this stage).
Let $k$ be the residue field of $A$
and $p \geq 0$ be its characteristic.
We put $\widehat \ZZ'= \prod_{l \not =p} \ZZ_l$.
Let $K$ be the fraction field of $A$, 
and let $K_s$ be a separable closure of $K$.
It determines a base point $\xi: \Spec(K) \to X$
so that we can deal with the Grothendieck fundamental 
group $\Pi_1(X, \xi)$ \cite{SGA1}. 

Let $(f_1, \dots, f_r)$ be a regular sequence of $A$
and consider the divisor $D= \sum D_i= \sum \mathrm{div}(f_i)$, it  has strict normal crossings.
We put $U = X \setminus  D= \Spec(A_D)$.

We recall that a  finite \'etale cover $V \to U$ is 
{\it tamely ramified} with respect to $D$ if the associated \'etale $K$--algebra
$L=L_1 \times \dots \times L_a$ is 
tamely ramified at the $D_i's$, that is, for each $i$, there exists $j_i$ such that 
for the  Galois closure $\widetilde L_{j_i}/K$ of $L_{j_i}/K$,
the inertia group associated to $v_{D_i}$ has order prime to 
$p$ \cite[XIII.2.0]{SGA1}.

Grothendieck and Murre defined the  
tame ({\it mod\'er\'e} in French) fundamental group 
$\Pi_1^D(U, \xi)$
with respect to  $U \subset X$ as defined in \cite[XIII.2.1.3]{SGA1} and \cite[\S 2]{GM}. This is a profinite quotient of $\Pi_1(U, \xi)$ whose quotients by open subgroups
 provides finite Galois tame  cover of $U$.
 
 \sm
 
Let $V \to U$ be a finite \'etale tame cover. 
In this case Abhyankar's lemma states that  
there exists a flat Kummer cover $X'=\Spec(A') \to X$
where $$A'= A[T_1,\dots, T_r]/ ( T_1^{n_1} - f_1, \dots, T_r^{n_r}- f_r)$$
and the $n_i$'s are coprime to $p$
such that $V'= V \times_X X' \to X'$ extends uniquely
to a finite \'etale cover $Y' \to X'$ \cite[XIII.5.2]{SGA1}.

\begin{slemma}\label{lem_picard} Let $V \to U$ be a finite \'etale cover 
which is tame. Then $\Pic(V)=0$.
\end{slemma}

\begin{proof} We use the same notation as above.
We know that  $X'$ is regular \cite[XIII.5.1]{SGA1}
so a fortiori locally factorial. It follows 
that  the restriction  maps $\Pic(X') \to \Pic(V') \to \Pic(V)$
are surjective \cite[21.6.11]{EGA4}. 
Since $A'$ is finite over the local ring $A$, it is semilocal so that $\Pic(A')=\Pic(X')=0$. Thus $\Pic(V)=0$ as desired.
\end{proof}

\smallskip

From now on we assume that $A$ is henselian.
According to \cite[18.5.10]{EGA4},  the finite 
$A$--ring $A'$ is a finite product of henselian local rings.
We observe  that $A' \otimes_A k= k[T_1,\dots, T_r]/ (T_1^{n_1}, \dots , T_r^{n_r})$ is a local Artinian algebra so that $A'$ is connected. 
It follows that $A'$ is a henselian local ring.
Its maximal ideal is $\gm'= \gm \otimes_A A' + \langle T_1, \dots, T_r \rangle$
so that $A' / \gm'= k$. Since there is an equivalence of categories
between finite \'etale covers of $A$ (resp.\, $A'$) and \'etale $k$--algebras \cite[18.5.15]{EGA4}, the base change from $A$ to $A'$ provides  an equivalence of categories
between the category of finite \'etale covers of $A$  and that of $A'$.

It follows that  $Y' \to X'$
descends uniquely to a  finite \'etale cover
$\widetilde f: \widetilde Y \to X$. 
From now on, we assume that 
 $V$ is furthermore connected, it implies that 
$$
H^0(V, \cO_V)=B[T_1,\dots, T_r]/ ( T_1^{n} - f_1, \dots, T_r^{n}- f_r)
$$
where $B$ is a finite connected \'etale cover of $A$. 
It follows that  $V \to U$ is  a quotient of a Galois cover of the shape
$$
B_n=B[T_1^{\pm 1},\dots, T_r^{\pm 1}]/ ( T_1^{n} - f_1, \dots, T_r^{n}- f_r)
$$
where $B$ is Galois cover of $A$ containing a  primitive $n$--th root of unity. 
We notice that $B_n$ is the  localization at $T_1\dots T_r$ of
$B_n'=B[T_1,\dots, T_r]/ ( T_1^{n} - f_1, \dots, T_r^{n}- f_r)$.
We have $$
\Gal(B_n/A_D)= \bigl( \prod_{i=1}^r \mu_n(B) \bigr) \rtimes \Gal(B/A).
$$
Passing to the limit we obtain an isomorphism
$$
\pi_1^t(U, \xi) \cong \bigl( \prod_{i=1}^r \widehat \ZZ'(1) \bigr)
\rtimes \pi_1(X, \xi).
$$
We denote by  $f: U^{sc, t} \to U$ the profinite \'etale cover
associated to the quotient $\pi_1^t(U, \xi)$ of $\pi_1(U,\xi)$.
According to \cite[thm. 2.4.2]{GM}, it is the universal tamely ramified cover of $U$. 
It is a localization of the inductive limit $\widetilde B'$ of the  $B_n'$. On the other hand we consider  the inductive limit $\widetilde B$ of the $B$'s and observe that $\widetilde B'$
is a  $\widetilde B$-ring.

\subsection{Blow-up}\label{subsec_blow_up}

We follow a blowing-up construction arising from  
 \cite[lemma 15.1.1.6]{EGA4}.
We denote by $\widehat X$ the blow-up of $X=\Spec(A)$ 
at his closed point, this is a regular scheme \cite[\S 8.1, 
th. 1.19]{Liu} and the exceptional divisor $E \subset \widehat X$ is 
a Cartier divisor isomorphic to $\mathbb{P}^{r-1}_k$.
We denote by  $R=\mathcal{O}_{\widehat X,\eta}$ the local ring at the generic point $\eta$ of $E$.  The ring $R$ is a DVR of
fraction field $K$ and of  residue field 
$F=k(E)=k(t_1, \dots, t_{r-1})$ where $t_i$ is the image
of $\frac{f_i}{f_r} \in R$ by the specialization map.
We denote by $v: K^\times \to \ZZ$
the discrete valuation associated  to $R$.

We deal now with a Galois extension $B_n$ of $A_D$ as above.
Since $B$ is a connected finite \'etale  cover of $A$,
$B$ is regular and  local; it is furthermore henselian \cite[18.5.10]{EGA4}.
We denote by $L$ the fraction field of $B$ and by $L_n$
that of $B_n$. We have $[L_n:L]= n^r$.
We want to extend the valuation $v$ to $L$ and to $L_n$.

We denote by $l=B/\gm_B$ the residue field of $B$, this is a finite Galois field  extension of $k$. Also $(t_1,\dots, t_r)$ is a system of parameters for
$B$. We denote by $w:L^\times \to \ZZ$ 
 the discrete valuation associated  to the
exceptional divisor of the blow-up of $\Spec(B)$ at its closed point. Then $w$ extends $v$ and $L_w/K_v$ is an unramified
extension of degree $[L:K]$ and of residual extension 
$F_l= l(t_1, \dots , t_{r-1})/k(t_1, \dots , t_{r-1})$.

On the other hand we denote by $w_n:L_n^\times \to \ZZ$ 
 the discrete valuation associated  to the
exceptional divisor of the blow-up of $\Spec(B_n)$ at its closed point. We put $l_n= B_n/ \gm_{B_n}$, 
we have $l=l_n$.
The valuation  $\frac{w_n}{n}$ on $L_n$ extends $w$ and 
its residual extension is 
$F_{l,n}= l\Bigl(  t_1^{1/n},\dots, t_{r-1}^{1/n} \Bigr) /k\bigl(t_1, \dots, t_{r-1} \bigr)$ 
so that $[F_{l,n}:F_l]=n^{r-1}$.
Furthermore the ramification index $e_n$ of $L_n/L$ is $\geq n$.
Since $n^r \leq e_n \, {[F_{l,n}:F_l]} \leq [L_n:K]=n^r$(where the last inequality is \cite[\S VI.3, prop. 2]{BAC}) it follows that $e_n=n$. 
The same statement shows that the map 
$L_w \otimes_L L_n \to L_{w_n}$ is an isomorphism.
To summarize $L_{w_n}/L_w$ is tamely ramified of ramification index $n$ and of degree $n^r$.
Altogether we have $L_{w_n}=L_w \otimes_K L_n$ so 
that $L_{w_n}$ is Galois over 
$K_v$ of group ${\prod_i \mu_n(B) \rtimes \Gal(B/A)}
= {\prod_i \mu_n(l) \rtimes \Gal(l/k)}$.

We denote by $\Delta$  the diagonal
embedding $\mu_n(l) \subset \prod_i \mu_n(l)$. 
We put $L_{w_n}^{\Delta}= L_n^{\Delta(\mu_n(B))}$.
Since $t_r$ is an uniformizing parameter of $K_v$ and 
since $\Delta(\zeta) \, . \, t_r= \zeta. t_r$ for
each  $\zeta \in \mu_n(B)$, it follows that 
$(L_{w_n})^{\Delta}$ is the maximal unramified extension of
$L_{w_n}/K_v$.

\subsection{Loop cocycles and loop torsors}

Let $G$ be an $X$--group scheme locally of finite presentation.
A loop cocycle is an element of $Z^1\bigl(\pi_1^t(U), G(\widetilde B) \bigr)$
and it defines a Galois cocycle in $Z^1(\pi_1^t(U), G(U^{sc, t}))$.
We denote by $Z^1_{loop}(\pi_1^t(U), G(U^{sc, t}))$ the image of 
the map  $Z^1\bigl(\pi_1^t(U), G(\widetilde B) \bigr) \to Z^1(\pi_1^t(U), G(U^{sc, t}))$ and by $H^1_{loop}(U, G)$ the image of 
the map  $$
Z^1\bigl(\pi_1^t(U), G(\widetilde B) \bigr) \to 
H^1(\pi_1^t(U), G(U^{sc, t})) \to H^1(U, G).
$$
We say that a $G$-torsor $E$ over $U$ (resp.\, an  fppf sheaf $G$-torsor)
is a loop torsor if its class belongs to
$H^1_{loop}(U, G) \subset H^1(U, G)$.

A given class $\gamma \in H^1_{loop}(U, G)$ is represented
by a $1$--cocycle $\phi: \Gal(B_n/A_D) \to G(B)$ for some cover $B_n/A$ as above.
Its restriction $\phi^{ar}: \Gal(B/A) \to G(B)$ to the subgroup $\Gal(B/A)$ of
$\Gal(B_n/A_D)$ is called the ``arithmetic part'' and the other restriction
$\phi^{geo}: \prod_i \mu_n(B) \to \gG(B)$ is called the geometric part.
We observe that $\phi^{geo}$ is  a $B$-group homomorphism.

Furthermore for $\sigma \in \Gal(B/A)$ and $\tau \in 
\prod_i \mu_n(B)$ the computation of \cite[page 16]{GP}
shows that $\phi^{geo}(\sigma \tau \sigma^{-1})=
\phi^{ar}(\sigma) \, {^\sigma \!  \phi}(\tau)  \, 
\phi^{ar}(\sigma)^{-1}$ so that $\phi^{geo}$
 descends to a homomorphism of $A$-group schemes 
$\phi^{geo}: \mu_n^r \to {_{\phi^{ar}}\!G}$.  
This provides a parameterization of loop cocycles.

\begin{slemma} \label{lem_dico}(1) For $B_n/A_D$ as above, the map $\phi \mapsto (\phi^{ar}, \phi^{geo})$
provides a bijection between $Z^1_{loop}\bigl( \Gal(B_n/A_D) ,  G(B) \bigr)  $ and the couples
$(z, \eta)$ where $z \in Z^1\bigl( \Gal(B/A) ,  G(B) )$ and 
$\eta: \prod_i \mu_n \to {_zG}$
is an $A$--group homomorphism.

\smallskip

\noindent (2) The map $\phi \mapsto (\phi^{ar}, \phi^{geo})$
provides a bijection between $Z^1_{loop}\bigl( \pi^1(U, \xi)^t ,  G(\widetilde B)\bigr)$ and the couples
$(z, \eta)$ where $z \in Z^1\bigl( \pi^1(X, \xi) , G(\widetilde B) )$ and $\eta: \prod_{i=1}^r \widehat \ZZ' \to {_zG}$ is an $A$--group homomorphism.

\end{slemma}

\begin{proof} This is similar with \cite[lemma 3.7]{GP}.
\end{proof}

We examine more closely the case of  a finite \'etale $X$--group scheme 
$\gF$ of constant degree $d$.

\begin{slemma} (1)  $\gF(\widetilde B)= \gF(X^{sc})= \gF( U^{sc,t})$.

\sm

\noindent (2) We assume that $d$ is prime to $p$.
We have $H^1_{loop}( U, \gF)= H^1(U, \gF)$.

\smallskip

\noindent (3) We assume that $d$ is prime to $p$.
Let $f: \gF \to \gH$ be a homomorphism of 
$A$--group schemes (locally of finite type).
Then $f_*\Bigl( H^1(U, \gF) \Bigr) \subset 
H^1_{loop}( U, \gH)$.
\end{slemma}

\begin{proof} 
\noindent (1) We are given a cover $B_n/A_D$  as above  such that $\gF_{B_n} \cong  \Gamma_{B_n}$ is finite constant.
as above. Since $B$ and $B_n$ are connected, the map $\gF(B) \to \gF(B_n)$ reads
as the identity $\Gamma \cong \gF(B) \to \gF(B_n) \cong \Gamma$ so is bijective.
By passing to the limit we get  $\gF(\widetilde B)= \gF( U^{sc,t})$.

\sm

\noindent (2) Let $\gE$ be a $\gF$--torsor over $U$. 
This is a finite \'etale $U$--scheme.
Since $U$ is noetherian and connected, 
we have a decomposition $\gE= V_1 \times_U \cdots \times_U V_l$ where each 
$V_i$ is a connected finite \'etale $U$--scheme of constant degree $d_i$.
We have $d_1+ \dots +d_l=d$ so that we can assume that $d_1$ is prime to $p$.
We have then $\gE(V_1) \not = \emptyset$. 

It follows that $f_1: V_1 \to U$ is a finite \'etale cover so that 
there exists a  factorization $ U^{sc, t} \to V_1 \xrightarrow{h} U$ of $f$
so that $\gE(U^{sc, t} ) \not = \emptyset$. Therefore $[\gE]$ arises
from $H^1(\pi_1^t(U, \xi), \gF(U^{sc, t})) \subset H^1(U, \gF)$.
It follows that $H^1(\pi_1^t(U, \xi), \gF(U^{sc, t})) \simlgr H^1(U, \gF)$.
We use now (1) and obtain the desired bijection  $H^1(\pi_1^t(U, \xi), \gF(B)) \simlgr H^1(U, \gF)$.

\sm

\noindent (3) This follows readily from (2).
\end{proof}

\subsection{Twisting by loop torsors}
We assume that the $A$--group scheme $G$ acts 
on an $A$--scheme $Z$.
Let $\phi: (\prod_i^r \mu_n)(B) \rtimes \Gal(B/A) \to G(B)$
be a loop cocycle.
It gives rise to  an $A$--action  of
 $\mu_n^r$ on $_{\phi^{ar}}\!Z$. 
We denote by $(_{\phi^{ar}}\!Z)^{\phi^{geo}}$
the fixed point locus for this action, it is representable
by a closed $A$--subscheme of $_{\phi^{ar}}\!Z$ \cite[A.8.10.(1)]{CGP}. We have a closed embedding  $(_{\phi^{ar}}\!Z)^{\phi^{geo}} \times_X U  \subset {_\phi \! Z}$ of
$U$-schemes.

\section{Fixed points method }

\begin{stheorem}\label{thm_fixed}
Let $X=\Spec(A)$ be a henselian regular local scheme  and $U=X \setminus D$  as above.
We denote by $v: K^\times \to \ZZ$
the discrete valuation associated  to the
exceptional divisor $E$ of the blow-up of $X$ at its closed point.

Let $G$ be an affine  $A$-group scheme  of finite presentation
acting on a proper smooth $A$--scheme $Z$.
Let $\phi$ be a loop cocycle for $G$. Then 
$Y=  \bigl( _{\phi^{ar}}Z\bigr)^{\phi^{geo}}$ is  a
smooth proper $A$--scheme and  
the following are equivalent:

\sm

(i) $(_\phi\!Z)(K_v) \not = \emptyset$; 

\sm

(ii) $Y(k) \not \not = \emptyset$;

\sm 

(iii)  $Y(U) \not \not = \emptyset$;

\sm 

(iv) $(_\phi\!Z)(U) \not = \emptyset$.

\end{stheorem}

This is quite similar with the fixed point theorem
\cite[\S , thm. 7.1]{GP}. The following example makes the connection.

\begin{sexample}{\rm We assume that  $A=k[[t_1,\dots, t_r]]$ for a field $k$
and \break $k[U]=k[[t_1,\dots, t_n]]\bigl[ \frac{1}{t_1}, \dots, \frac{1}{t_r} \bigr]$.  
We consider  an affine algebraic  $k$--group $G$ acting on a smooth proper
$k$--scheme $Z$. 
In this case $K=k((t_1,\dots, t_r))$ and  $A$ embeds in $k\bigl(\frac{t_1}{t_r}, \dots,  \frac{t_{r-1}}{t_r}\bigr)[[t_r]]$
so that $K$ embeds in $k\bigl(\frac{t_1}{t_n}, \dots,  \frac{t_{r-1}}{t_r}\bigr)((t_r))$
which is nothing but the complete field $K_v$. 
If $Q$ is a loop $G$-torsor over $U$, 
the statement is then that  ${^QZ}(U) \not = \emptyset$ if and only if ${^QZ}(K_v) \not = \emptyset$.
Taking a cocycle $\phi \in Z^1( \pi_1(U)^t, G(k_s))$ for $E$, 
this rephrases by the equivalence between  $(_{\phi}Z)(U) \not = \emptyset$ and $(_{\phi}Z)(K_v) \not =\emptyset$. 

What we have from  \cite[thm. 7.1]{GP} (in characteristic zero but this extends to this tame setting) is 
the equivalence between $(_{\phi}Z)(k[t_1^{\pm 1}, \dots, t_r^{\pm 1}]) \not = \emptyset$ and  \break $(_{\phi}Z)\bigl(k((t_1)) \dots ((t_r)) \bigr)\not = \emptyset$. 
Since  $(_{\phi}Z)(k[t_1^{\pm 1}, \dots, t_r^{\pm 1}]) \subset  (_{\phi}Z)(U)$
and  $(_{\phi}Z)\bigl( K_v \bigr) \subset  (_{\phi}Z)\bigl( k((t_1)) \dots ((t_r)) \bigr) $, it follows
that this special case of Theorem \ref{thm_fixed}  is a consequence of the fixed point result of \cite{GP}.
 }
\end{sexample}

We proceed to the proof of Theorem \ref{thm_fixed}.

\begin{proof} 
According to \cite[A.8.10.(1)]{CGP}, $Y=  \bigl( _{\phi^{ar}}(Z^{\phi^{geo}}) \bigr)$ is  a closed  $A$--scheme of 
$_{\phi^{ar}}Z$ so it is proper.
It is smooth over $X$  according to point (2) of the same reference.  
Let  $\phi: \Gal(B_n/A_D) \to G(B)$ be the  loop $1$-cocycle
 for some Galois cover $B_n/A_D$ as above for  some $n$ prime to $p$.
 Up to replacing $G$ by $_{{\phi^{ar}}\!G}$ and
$G$ by $_{\phi^{ar}Z}$, we can assume
that $\phi^{ar}=1$ without lost of generality.

\sm

\noindent $(ii) \Longrightarrow (iii)$. Since $Y_k$ is the special fiber of the smooth $X$--scheme
$Y$, Hensel's lemma shows that $Y(A) \to Y(k)$ is onto.
Since $Y(k)$ is not empty, it follows that $Y(A)$ is not empty and so 
is $Y(U)$.

\sm

\noindent $(iii) \Longrightarrow (iv)$.
Since $Y(U) \subset {_\phi\!Z}(U)$, $Y(U) \not = \emptyset$ implies that 
${_\phi\!Z}(U)\not = \emptyset$.

\sm

\noindent $(iv) \Longrightarrow (i)$. This is obvious.

\sm

\noindent $(i) \Longrightarrow (ii)$.
We assume that 
$(_\phi\!Z)(K_v) \not = \emptyset$.
By definition we have
$$
(_\phi Z)(K_v) = \bigl\{ z \in Z(L_{w_n})
\, \mid \, \phi(\sigma). \sigma(z)=z  \enskip
\forall \sigma  \in \Gal(L_n/K) \bigr\}
$$
and our assumption is that this set is non-empty.
Let $O_{w_n}$ be the valuation ring of 
$Z(L_{w_n})$. Since $Z$ is proper over $X$, 
we have a specialization map
$Z(L_{w_n}) = Z(\cO_{w_n}) \to Z_k(F_{l,n})$. 
We get that the set 
$$
\bigl\{ z \in Z_k(F_{l,n}) \mid \, \phi(\sigma). \sigma(z)=z  \enskip
\forall \sigma  \in \Gal(L_{w_n}/K_v) \bigr\}
$$
is not empty. Since we have an embedding
$$
F_{l,n}=l\bigl( t_1^{1/n} ,\dots, t_{r-1}) \enskip  \hookrightarrow
\enskip l\bigl( \bigr(  t_1^{1/n} \bigr)\bigr) \dots
\bigl( \bigl(  t_{r-1}^{1/n} \bigr)\bigr) 
$$
in a higher field of Laurent series,
successive specializations along the coordinates 
$t_1^{1/n}$, ...,   $t_{r-1}^{1/n}$ 
show similarly  that  the set 

\begin{equation} \label{eq_omega}
\Bigl\{ z \in (Z_k)\bigl( l \bigr)
\,  \mid \, \phi(\sigma). \sigma(z)=z  \enskip
\forall \sigma  \in   \Gal(L_{w_n}/K_v) \Bigr\}
\end{equation}
is not empty. Since $\eta^{ar}=1$, this set 
is $(Z_k)^{\eta^{geo}}(k)$. Thus $Y(k)=(Z_k)^{\eta^{geo}}(k)$ is non empty.
\end{proof}

\section{Parabolic subgroups of loop reductive group schemes}

\subsection{Chevalley groups}
Let $G_0$ be Chevalley group defined over $\ZZ$.
 Let $T_0$ be a maximal split $\ZZ$-subtorus
 of $G_0$ together with a Borel subgroup $B_0$ containing it. 
  We denote by $\Delta_0$ the Dynkin diagram
 of $(G_0,B_0,T_0)$.
We denote by $G_{0,ad}$ the adjoint quotient of 
$G_0$ and by $G_0^{sc}$ the simply connected covering of $DG_0$. 
We have a map $\Aut(G_0) \to \Aut(G_0^{sc}) \simlgr \Aut(G_{0,ad})$ 
and a fundamental exact sequence
$$
1 \to G_{0,ad} \to \Aut(G_{0,ad}) \to \Out(G_{0,ad}) \to 1
$$
 where $\Out(G_{0,ad}) \simlgr \Aut(\Delta_0)$.
 We recall that there is a bijection 
 $I \to P_{0,I}$ between the finite subsets of 
 $\Delta_0$ and the parabolic subgroups of $G_0$ containing $B_0$ \cite[XXVI.3.8]{SGA3}; it is increasing for the inclusion order, in particular $B_0= P_{0,\emptyset}$ and $G_0= P_{0,\Delta_0}$.
We consider the  total scheme $\mathrm{Par}_{G_0}$ of parabolic subgroups
of $G_0$, it is a projective smooth $\ZZ$--scheme
equipped with a type map $\mathbf{t}:  \mathrm{Par}_{G_0} \to \Ouf(\Delta_0)$
where $\Ouf(\Delta_0)$ stands for the finite constant scheme attached
to the set of subsets of $\Delta_0$
\cite[XXVI.3]{SGA3}.
 The fiber at $I$ is  denoted by $\mathrm{Par}_{G_0,I}$, 
 it has connected fibers and is the scheme of parabolic subgroups
 of $G_0$ of type $I$. 
 We have a natural  action of $\Aut(G_0)$ on $\mathrm{Par}_{G_0}$.
As in \cite[\S 5.1]{Gi}, we denote by $\Aut_I(G_0)$ the
stabilizer of $I$ for this action. By construction
 $\Aut_I(G_0)$ acts on $\mathrm{Par}_{G_0,I}$.

 \subsection{Definition}
 Let $G$ be a reductive $U$-group scheme in the sense of Demazure-Grothendieck  \cite[XIX]{SGA3}. Since $U$ is connected and $G$ is locally splittable \cite[XXII.2.2]{SGA3}
for the \'etale topology,  $G$ is an \'etale form 
of a Chevalley group $G_0$  as above defined over $\ZZ$.

We say that $G$ is a {\it loop group scheme} if the
 $\Aut(G_0)$-torsor $Q=\Isom(G_0,G)$ 
(defined in \cite[XXIV.1.9]{SGA3}) is a loop $\Aut(G_0)$-torsor.
We denote by $G_{0,ad}$ the adjoint quotient of 
$G_0$ and by  $G_0^{sc}$ the simply connected covering of $DG_0$. 
We have a map $\Aut(G_0) \to \Aut(G_0^{sc}) \simlgr \Aut(G_{0,ad})$
which permits to see $G_{ad}$ (resp.\ $G^{sc}$) as twisted forms
of  $G_{0,ad}$ (resp.\ $G^{sc}_0$) so that $G_{ad}$ and $G^{sc}$ are 
also loop reductive group schemes.
We consider the map $\Aut(G_0) \to \Aut(G_{0,ad}) \to
\Out(G_{0,ad}) \simlgr \Aut(\Delta_0)$.

If $\phi: \Gal(B_n/A_D) \to \Aut(G_0)(B)$ is a loop cocycle, 
we get an action of $\Gal(B_n/A_D)$  on $\Delta_0$ called the star action.
If $I$ is stable under the star action, we can twist 
$\mathrm{Par}_{G_0,I}$ by $\phi$ and deal with the 
scheme ${_\phi\!\mathrm{Par}_{G_0,I}}$ which is the scheme of
parabolic subgroup schemes of $G$ of type $I$.

\subsection{Parabolics} 

\begin{stheorem}\label{thm_parabolic} Assume that $G$ is a loop reductive $U$-group scheme and let $\phi: \Gal(B_n/A_D) \to \Aut(G_0)(B)$ 
be a loop cocycle such that $G \cong {_\phi \! G_0}$. 
Let $I \subset \Delta_0$ be a subset stable under the star action 
defined by $\phi$. Then the following are equivalent:

\sm

(i) $G$ admits a $U$--parabolic subgroup of type $I$;

\sm

(ii) the $k$--morphism 
$\eta^{geo}_k: \mu_n^r \to 
\Aut(_{\eta^{ar}}\! G_0)_k
= \bigl({_{\eta^{ar}} \!\Aut(G_0)}\bigr)_k$ 
normalizes a parabolic $k$--subgroup of $_{\eta^{ar}}G_{0,k}$ of type $I$;

\sm 

(iii)  $G_{K_v}$ admits a parabolic subgroup of type 
$I$.
\end{stheorem}

\begin{proof}
Without loss of generality we can assume that 
$G$ is adjoint. 
Our assumption on the star action is rephrased by saying that 
$\phi$ takes values in $\Aut_I(G_0)$.
We apply Theorem \ref{thm_fixed} to 
the action of $\Aut_I(G_0)$ on the 
proper $A$-scheme $\mathrm{Par}_{G_0,I}$.
We consider the $A$-scheme $Y= ({_{\phi^{ar}}}\!\mathrm{Par}_{G_0,I})^{\phi^{geo}}$.
Theorem \ref{thm_fixed} shows that the following statements
are equivalent.

\smallskip

(i') $(_\phi\!\mathrm{Par}_{G_0,I})(U) \not = \emptyset$;

\smallskip

(ii') $Y(k) \not = \emptyset$.

\smallskip

(iii') $(_\phi\!\mathrm{Par}_{G_0,I})(K_v) \not = \emptyset$.

\smallskip

Clearly (i') is equivalent to condition (i) of the Theorem
and similarly we have $(iii') \Longleftrightarrow (iii)$. 
It remains to establish the equivalence between (ii) and (ii').

Assume that  $({_{\phi^{ar}}}\!\mathrm{Par}_{G_0,I})^{\phi^{geo}}(k)$
is not empty and pick a $k$--point $z$. 
Then the stabilizer $({_{\phi^{ar}}}\!G_0)_z$ is a $k$--parabolic subgroup of ${_{\phi^{ar}}}\!G_0$ of type $I$
which is stabilized by the action $\phi^{geo}_k$.
In other words,  $\phi^{geo}_k$ normalizes $({_{\phi^{ar}}}\!G)_z$.
Conversely we assume that ${_{\phi^{ar}}}\!G$  
admits a $k$--parabolic subgroup of type $I$ 
normalized by $\phi^{geo}$. It defines then a point $z \in 
({_{\phi^{ar}}}\!\mathrm{Par}_{G_0,I})(k)$ which is 
fixed by $\phi^{geo}$.

\end{proof}

\subsection{An example}

Assume that the residue field $k$ is not of characteristic $2$ and 
consider the diagonal quadratic form of dimension $2^r$
$$
q = \sum_{I \subset \{1, \dots ,r\}}  u_I \,  t^I (x_I)^2 
$$
where $t_I= \prod_{i \in I} t_i$ and $u_I \in A^\times$.
Then $\SO(q)$ is a loop reductive group scheme over $U$.
Since the projective quadric $\{q =0 \}$ is a scheme of parabolic subgroups
of $\SO(q)$, Theorem \ref{thm_parabolic} shows that 
$q$ is isotropic over $A_D$ if and only if 
$q$ is isotropic over $K_v$.
The $2$-dimensional case is related with \cite[proof of Theorem 3.1]{CTPS}.

\bigskip

\medskip


\bigskip
 
 
\begin{thebibliography}{99}



\bibitem{Bo} A. Borel, {\it  Linear algebraic groups}, 2nd edn, 
Graduate Texts in Mathematics {\bf 126} (Springer, New York, 1991).





\bibitem{BAC} N. Bourbaki, {\it Alg\`ebre commutative}, Ch. 1 \`a  10, Springer.

    



 \bibitem{CTPS} J.-L. Colliot-Th\'el\`ene,  R. Parimala, V. Suresh, 
 {\it Patching and local-global principles for
 homogeneous spaces over function fields of $p$-adic curves},
 Comment. Math. Helv. {\bf 87} (2012), 1011-1033. 
 

\bibitem{CGP} B. Conrad, O. Gabber, G. Prasad, {\it
Pseudo-reductive groups},   Cambridge University Press, second edition (2016).


 





\bibitem{EGA4} A.\ Grothendieck (avec la collaboration de J.\ Dieudonn\'e), 
{\it El\'ements de G\'eom\'etrie Alg\'ebrique IV}, Publications math\'ematiques de l'I.H.\'E.S. 
no 20, 24, 28 and 32 (1964 - 1967).




\bibitem{Gi} P. Gille, {\it Sur la classification des 
sch\'emas en groupes semi-simples}, ``Autour des sch\'emas en groupes, III'', 
Panoramas et Synth\`eses {\bf 47} (2015), 39-110.


\bibitem{GP} P. Gille and A. Pianzola, {\it  Torsors, reductive group schemes and extended affine Lie algebras},
Memoirs of  AMS {\bf  1063} (2013). 
 
 
 \bibitem{GM} A. Grothendieck, J. P. Murre, 
 {\it The tame fundamental group of a 
 formal neighbourhood of a divisor with normal crossings on a scheme}, Lecture Notes in
Mathematics {\bf 208} (1971), Springer-Verlag, Berlin-New York.
 
 

 
\bibitem{Liu} Q. Liu, {\it  Algebraic geometry and arithmetic curves},  Oxford Graduate Texts in Mathematics {\bf  6} (2002),  Oxford University Press, Oxford. 
    
 


\bibitem{SGA1} {\it S\'eminaire de G\'eom\'etrie alg\'ebrique de l'I.H.E.S.,  Rev\^etements \'etales et groupe fondamental, dirig\'e
par  A. Grothendieck},  Documents math\'ematiques vol. 3 (2003), Soci\'et\'e math\'ematique de France.

\bibitem{SGA3} {\it S\'eminaire de G\'eom\'etrie alg\'ebrique de l'I.\ H.\ E.\ S., 1963-1964,
sch\'emas en groupes, dirig\'e par M.\ Demazure et A.\ Grothendieck},  
Lecture Notes in Math. 151-153. Springer (1970).




  




\end{thebibliography}
\end{document}